\documentclass[12pt]{amsart}
\usepackage[all, cmtip]{xy}

\usepackage{times,amsfonts,amsmath,amstext,amsbsy,amssymb,
amsopn,amsthm,upref,eucal}

\DeclareMathOperator{\erg}{erg}

\newcommand \mtrho {{\mathfrak M}({\mathfrak T}, \rho)}

\newcommand \mergtrho {{\mathfrak M}_{\erg}({\mathfrak T}, \rho)}

\newcommand \mergtrhof {{\mathfrak M}_{\erg}({\mathfrak T}, \rho_f)}

\DeclareMathOperator{\mes}{mes}
\DeclareMathOperator{\Mes}{Mes}

\DeclareMathOperator{\Int}{Int}

\DeclareMathOperator{\inv}{inv}

\DeclareMathOperator{\myint}{int}

\newtheorem{theorem}{Theorem}
\newtheorem{lemma}{Lemma}
\newtheorem{proposition}{Proposition}
\newtheorem{corollary}{Corollary}

\begin{document}

\title[Ergodic Decomposition for Inductively Compact Groups]{Ergodic Decomposition for Measures Quasi-Invariant Under Borel Actions of Inductively Compact Groups}
\author{Alexander I. Bufetov}

\address{Steklov Institute of Mathematics,
Moscow}
\address{Institute for Information Transmission Problems,
 Moscow}
\address{National Research University Higher School of Economics,
 Moscow}
\address{Rice University, Houston TX}

\date{}

\begin{abstract}

The aim of this paper is to prove ergodic decomposition theorems for probability measures 
quasi-invariant under Borel actions of inductively compact groups (Theorem \ref{ergdecstrcont})  as well as for
$\sigma$-finite invariant measures (Corollary \ref{infdec}).
 For infinite measures the ergodic decomposition is not unique, but 
the measure class of the decomposing measure on the space
of projective measures is uniquely defined by the initial invariant measure
(Theorem \ref{ergdecinfpcl}).
\end{abstract}

\maketitle

\section{Introduction.}

\subsection{Outline of the main results.}

The first result of this paper
establishes existence and uniqueness of ergodic decomposition for probability measures quasi-invariant under
Borel actions of inductively compact groups (Theorem \ref{ergdecstrcont}). First we show 
in Proposition \ref{equivergindec} that for actions of inductively compact group 
ergodicity of a quasi-invariant measure is equivalent to its indecomposability
(as Kolmogorov's example \cite{Fomin} shows, this equivalence does not hold for measure-preserving actions of general groups). 
The ergodic decomposition is then constructed under the additional assumption that the Radon-Nikodym cocycle of the measure is
continuous in restriction to each orbit of the group
(the {\it fibrewise continuity} condition). This condition is only restrictive  for actions of uncountable groups.
The proof of Theorem \ref{ergdecstrcont} relies on Rohlin's method of constructing ergodic
decompositions.

Theorem \ref{ergdecstrcont} is then applied to $\sigma$-finite invariant measures.
In this case the ergodic decomposition is not unique. The measure class of the decomposing measure on the space
of projective measures is however uniquely defined by the initial invariant measure
(Theorem \ref{ergdecinfpcl}).
In the sequel \cite{bufetov-inferg} to this paper, its results are applied to the 
ergodic decomposition of infinite Hua-Pickrell measures, introduced by Borodin and Olshanski \cite{BO}, 
on spaces of infinite Hermitian matrices.

For completeness of the exposition, Kolmogorov's example
of a group action admitting decomposable ergodic measures is also included.

For actions of the group ${\mathbb Z}$ with a quasi-invariant measure,
the ergodic decomposition theorem
was obtained by Kifer and Pirogov \cite{Kif} who used the method of Rohlin \cite{Roh}.

For actions of locally compact groups, a general ergodic decomposition theorem is due
to Greschonig and Schmidt \cite{GreSchm} whose approach is based on Choquet's theorem (see, e.g., \cite{Phelps}).
In order to be able to apply Choquet's theorem, Greschonig and Schmidt use Varadarajan's theorem \cite{Var} claiming
that every Borel action of a locally compact group admits a continuous realization (see Theorem \ref{varad} below).
It is not clear whether  a similar result holds for inductively compact groups (see the question following Theorem \ref{varad}).

For the natural action of the infinite unitary group on the space of infinite Hermitian matrices,
ergodic decomposition of invariant probability measures was constructed by Borodin and Olshanski \cite{BO}.
Borodin and Olshanski \cite{BO} rely on Choquet's Theorem, which, however, cannot be used directly
since the space of infinite Hermitian matrices is not compact. Borodin and Olshanski
embed the space of probability  measures on the space of infinite Hermitian matrices into a larger convex compact
metrizable set to which Choquet's Theorem can be applied.

Rohlin's approach to the problem of ergodic decomposition requires neither continuity nor compactness,
and the results of this paper apply to all Borel actions of inductively compact groups.
The martingale convergence theorem is used instead of the ergodic theorem on which Rohlin's argument relies;
the idea of using martingale convergence for studying invariant measures for actions of inductively compact
groups goes back to Vershik's note \cite{Vershik}.

\subsection{Measurable actions of topological groups on Borel spaces.}
\subsubsection{Standard Borel spaces}
Let $X$ be  a set, and let $\mathcal{B}$ be a sigma-algebra on $X$.
The pair $(X, \mathcal{B})$ will be called {\it{a standard Borel space}} if
there exists a bijection between $X$ and the unit interval  which sends $\mathcal{B}$ to the sigma-algebra of Borel sets. We will continue to call $\mathcal{B}$ the Borel sigma-algebra, and measures defined on
$\mathcal{B}$ will be called Borel measures.

Let $\mathfrak M(X)$ be the space of Borel probability measures on $X$. A natural $\sigma$-algebra $\mathcal B ( \mathfrak M(X) )$ on the space $\mathfrak M(X)$ is defined as follows. Let $A \in X$ be a Borel subset, let $\alpha \in \mathbb R$, and let
\[
M_{A, \alpha} = \big\{ \nu \in \mathfrak M(X): \, \nu(A) > \alpha \big\}.
\]
The $\sigma$-algebra $\mathcal B (\mathfrak M(X))$ is then the smallest $\sigma$-algebra containing all sets $M_{A,\alpha}$, $A \in \mathfrak B(X)$, $\alpha \in \mathbb R$. Clearly, if $(X, \mathcal B)$ is a standard Borel space, then $(\mathfrak M(X), \mathcal B(\mathfrak M(X))$ is also a standard Borel space.

A Borel measure $\nu$ on a standard Borel space $(X, \mathcal B)$ is called {\it $\sigma$-finite} if there exists a countable family of disjoint Borel subsets
\[
X_1, \,X_2,\, \ldots,\, X_n,\ldots
\]
of $X$ such that
\[
X = \bigcup _{n=1}^\infty X_n
\]
and such that $\nu(X_n) < +\infty$ for any $n \in \mathbb N$. We denote by $\mathfrak M^\infty (X)$ the space of all $\sigma$-finite Borel measures on $X$ (note that, in our terminology, finite measures are also $\sigma$-finite). The space $\mathfrak M^\infty (X)$ admits a natural Borel structure: the Borel $\sigma$-algebra is generated by sets of the form
\[
\left\{ \nu \in \mathfrak M^\infty(X): \, \alpha< \nu(A)< \beta \right\},
\]
where $\alpha, \beta$ are real and $A$ is a Borel subset of $X$.

If $\nu$ is a Borel measure on $X$ and $f\in L_1(X, \nu)$, then  for brevity we denote
$$
\nu(f)=\int\limits_X fd\nu.
$$

As usual, by {\it a measure class} we mean the family of all sigma-finite Borel measures with the same sigma-algebra of sets of measure zero.
The measure class of a measure $\nu$ will be denoted $[\nu]$.
We write $\nu_1 \ll \nu_2$ if $\nu_1$ is absolutely continuous with respect to $\nu_2$, while
the notation $\nu_1 \perp \nu_2$ means, as usual, that the measures $\nu_1, \nu_2$ are mutually singular.

\subsubsection{Measurable actions of topological groups}

Now let $G$ be a topological group endowed with the Borel sigma-algebra.
Assume that the group $G$ acts
on $X$ and for $g \in G$ let $T_g$ be the corresponding transformation. The action will be called {\it measurable} (or {\it Borel}) if the map
\[
\mathfrak{T}:G \times X \to X, \quad \mathfrak{T}(g,x) = T_g x
\]
is Borel-measurable.
The group $G$ acts on $\mathfrak M(X)$. It will be convenient for us to consider the right action and for $g \in G$ to introduce the measure
\[
\nu \circ T_g (A) = \nu (T_g A).
\]
The resulting right action is, of course, Borel.

\subsubsection{Inductively compact groups}
Let
\[
K(1) \subset K(2) \subset \, \ldots \, \subset K(n) \subset \, \ldots
\]
be an ascending chain of metrizable compact groups and set
\[
G = \bigcup\limits_{n=1}^\infty K(n).
\]
The group $G$ will then be called {\it inductively compact}. Natural examples are the infinite symmetric group
\[
S(\infty) = \bigcup_{n=1}^\infty S(n)
\]
or the infinite unitary group
\[
U(\infty) = \bigcup_{n=1}^\infty U(n)
\]
(in both examples, the inductive limit is taken with respect to the natural inclusions).

An inductively compact group $G$ is endowed with the natural topology of the inductive limit, under which a function on $G$ is continuous if and only if it is continuous in restriction to each $K(n)$.
The Borel $\sigma$-algebra on $G$ is the span of the Borel $\sigma$-algebras on $K(n), \, n \in \mathbb N$.

\subsection{Cocycles and measures}
\subsubsection{Measurable cocycles.}
 In this paper, a measurable cocycle over a measurable action $\mathfrak T$ of a topological group $G$ will always mean a positive real-valued multiplicative cocycle, that is, a measurable map
\[
\rho: G \times X \rightarrow \mathbb R_{>0}
\]
satisfying the cocycle identity
\[
\rho(gh,\, x) = \rho(g,\, T_hx) \cdot \rho(h,\,x).
\]
Given  a positive real-valued multiplicative
cocycle $\rho$ over a measurable action $\mathfrak T$ of a topological group $G$,  introduce the space $\mathfrak M({\mathfrak T}, \rho) \subset \mathfrak M(X)$ of Borel probability measures with Radon-Nikodym cocycle $\rho$ with respect to the action $\mathfrak T$:
\[
\mathfrak M({\mathfrak T}, \rho) = \left\{ \nu \in \mathfrak M (X): \, \frac{d\nu \circ T_g}{d\nu} (x) = \rho(g, x) \text{ for all } g \in G \text{ and } \nu \text{-almost all } x \in X \right\}.
\]
Note that for a given probability measure $\nu$, quasi-invariant under the action $\mathfrak T$, its Radon-Nikodym cocycle is not uniquely, but only almost uniquely defined: if two Radon-Nikodym cocycles $\rho_1,\, \rho_2$ corresponding to the same measure $\nu$ are given, then for any $g \in G$ the equality
\[
\rho_1 (g,x) = \rho_2 (g,x)
\]
holds for $\nu$-almost all $x \in X$.

Nonetheless, the space $\mathfrak M(\mathfrak T, \rho)$ is a convex cone. Indeed, if
\[
\nu_i \circ T_g (A) = \displaystyle\int\limits_A \rho (g,x) \, d\nu_i, \quad i=1,2
\]
then also
\[
(\nu_1+\nu_2) \circ T_g(A) = \displaystyle\int\limits_A \rho (g,x) \, d(\nu_1 + \nu_2).
\]
\subsubsection{Indecomposability and ergodicity}
As before, let $\rho$ be a  positive real-valued multiplicative measurable cocycle
over a measurable action $\mathfrak T$ of a topological group $G$ on a standard Borel space $(X, {\mathcal B})$.

A measure $\nu \in \mathfrak M(\mathfrak T, \rho)$ is called {\it indecomposable} in $\mathfrak M(\mathfrak T, \rho)$ if the equality $\nu = \alpha \nu_1 + (1-\alpha) \nu_2$, with $\alpha \in (0,1)$, $\nu_1, \nu_2 \in \mathfrak M(\mathfrak T, \rho)$ implies $\nu = \nu_1 = \nu_2$.

Recall that a Borel  set $A$ is called {\it almost invariant} with respect to a Borel measure $\nu$ if
for every $g \in G$ we have $\nu(A \triangle T_g A)=0$.
Indecomposability can be equivalently reformulated as follows.
\begin{proposition}
\label{twoindec}
A Borel probability measure $\nu\in \mathfrak M(\mathfrak T, \rho)$
is indecomposable in $\mathfrak M(\mathfrak T, \rho)$ if and only if  any Borel set $A$, almost-invariant under the
action ${\mathfrak T}$ with respect to the measure $\nu$, satisfies either $\nu(A)=0$ or $\nu(X\setminus A)=0$.
\end{proposition}

 A measure $\nu \in \mathfrak M(\mathfrak T, \rho)$ is called {\it ergodic} if for every $G$-invariant Borel set $A$ we have either $\nu(A)=0$ or $\nu(X\setminus A) =0$. The set of all ergodic measures with Radon-Nikodym cocycle $\rho$ is denoted $\mathfrak M_{\erg}(\mathfrak T, \rho)$.

Indecomposable measures are a fortiori ergodic. For actions of general groups, ergodic probability measures may fail to be indecomposable: as Kolmogorov showed, the two notions are different for the natural action of the group of
all bijections of $\mathbb Z$ on the space of bi-infinite binary sequences (for completeness, we recall Kolmogorov's example in the last Section). An informal reason is that actions of ``large'' groups may have ``too few'' orbits
(a countable set in Kolmogorov's example), and consequently a convex combination of distinct ergodic measures
may also be ergodic.

Nevertheless, for actions  of  inductively compact groups,
the two notions coincide:
\begin{proposition}\label{equivergindec}
Let ${\mathfrak T}$ be a measurable action of an inductively compact group $G$
on a standard Borel space $(X, {\mathcal B})$, and let $\rho$ be a positive measurable multiplicative cocycle over
${\mathfrak T}$. If a measure $\nu\in {\mathfrak M}({\mathfrak T}, \rho)$ is ergodic, then $\nu$ is indecomposable
in ${\mathfrak M}({\mathfrak T}, \rho)$.
\end{proposition}

\subsection{Ergodic decomposition of
quasi-invariant probability measures.}
\subsubsection{Fibrewise continuous cocycles}
To formulate the ergodic decomposition theorem for quasi-invariant measures, we
need additional assumptions on the Radon-Nikodym cocycle $\rho$.

Let $\mathfrak T$ be a measurable action of a topological group $G$
on a standard Borel space $(X, \mathcal B)$.

{\bf{Definition.}} A positive real-valued measurable cocycle $\rho: G\times X \to \mathbb R_{>0}$ over the action $\mathfrak T$ will be called {\it fibrewise continuous} if for any $x \in X$ the function $\rho_x: G \to \mathbb R_{>0}$ given by the formula $\rho_x (g) = \rho(g, x)$ is continuous.

{\bf{Remark.}} If $G$ is inductively compact,
\[
G = \bigcup\limits_{n=1}^\infty K(n), \quad K(n) \subset K(n+1)
\]
then, by definition of the inductive limit topology, the requirement of fibrewise continuity precisely means that for any $n \in \mathbb N$ the function $\rho_x$ defined above is continuous in restriction to $K(n)$.

For general actions  of topological groups, it is not clear whether the set of measures with a given Radon-Nikodym
cocycle is Borel. That is the case, however, for actions of inductively compact groups and fibrewise continuous cocycles:

\begin{proposition}
\label{rhoborel}
Let $\rho$ be a fibrewise continuous cocycle
over a measurable action $\mathfrak T$ of a separable metrizable group $G$ on a standard Borel space $(X, {\mathcal B})$.
Then the set $\mathfrak M(\mathfrak T, \rho)$ is a Borel subset of $\mathfrak M(X)$.
\end{proposition}

Indeed, for fixed $g \in G$ the set
\[
\left\{ \nu \in \mathfrak M (X): \frac{d\nu \circ T_g}{d \nu} = \rho (g,x) \right\}
\]
is clearly Borel. Choosing a countable dense subgroup in $G$, we obtain the result.

In Proposition \ref{ergborel} below, we shall see that for a measurable  action of an inductively compact group, the set of
ergodic measures with a given fibrewise continuous Radon-Nikodym cocycle is Borel as well.

\subsubsection{Integrals over the space of measures}
Let ${\tilde \nu}\in {\mathfrak M}({\mathfrak M}(X))$, in other words, let ${\tilde \nu}$ be  a Borel probability
measure on the space of Borel probability measures on $X$.
Introduce a measure $\nu\in {\mathfrak M}(X)$ by the formula
\begin{equation}
\label{dec1}
\nu=\displaystyle\int\limits_{{\mathfrak M}(X)}\eta d{\tilde \nu}(\eta).
\end{equation}

The integral in the right-hand side of (\ref{dec1}) is understood in the following weak sense. For any Borel set $A\subset X$, the function ${\rm int}_A: {\mathfrak M}(X)\to {\mathbb R}$ given by the formula  ${\rm int}_A(\eta)=\eta(A)$ is clearly Borel measurable. The equality (\ref{dec1}) means that for any Borel set $A\subset X$ we have
\begin{equation}
\label{dec2}
\nu(A)=\displaystyle\int\limits_{{\mathfrak M}(X)}\eta(A) d{\tilde \nu}(\eta).
\end{equation}

\subsubsection{The ergodic decomposition theorem}

\begin{theorem} \label{ergdecstrcont}
Let $\mathfrak T$ be a measurable action of an inductively compact group $G$ on a standard Borel space 
$(X, \mathcal B)$. Let $\rho$ be a fibrewise continuous positive real-valued multiplicative cocycle over $\mathfrak T$. 
There exists a Borel subset $\widetilde{X}\subset X$ and a surjective Borel map 
$$
\pi:\:\widetilde{X}\rightarrow\mathfrak{M}_{erg}\left(\rho,\mathfrak{T}\right)
$$ 
such that
\begin{enumerate}

\item For any $\eta\in\mathfrak{M}_{erg}(\rho,\mathfrak{T})$ we have $\eta\left(\pi^{-1}(\eta)\right)=1$,

\item For any $\nu\in\mathfrak{M}\left(\rho,\mathfrak{T}\right)$ we have
$$\nu=\int\limits_{\mathfrak{M}_{erg}\left(\rho,\mathfrak{T}\right)}\eta\;d\,\bar{\nu}(\eta)\;,$$
where $\bar{\nu}=\pi_*\,\nu$.
In particular, for any $\nu\in\mathfrak{M}\left(\rho,\mathfrak{T}\right)$ we have $\nu(\widetilde{X})=1$.

\item
The correspondence $\nu \to \overline \nu$ is a Borel isomorphism between Borel spaces
$\mathfrak M (\mathfrak T, \rho)$ and $\mathfrak M( \mathfrak M_{\erg} (\mathfrak T, \rho))$, and if $\nu \in \mtrho$ and 
$\tilde \nu \in {\mathfrak M}(\mergtrho)$ are such that we have\[
\nu = \displaystyle\int\limits_{\mergtrho} \eta\, d\tilde \nu (\eta),
\]
then $\tilde \nu= \overline \nu$.

\item For any $\nu_1, \nu_2\in { \mathfrak M}_{\erg} (\mathfrak T, \rho)$, we have $\nu_1\ll\nu_2$ if and only if
$\overline\nu_1\ll\overline\nu_2$, and $\nu_1\perp \nu_2$ if and only if  $\overline\nu_1 \perp \overline\nu_2$.

\end{enumerate}

\end{theorem}

\subsection{Ergodic Decomposition of Infinite Invariant Measures}
\subsubsection{Reduction to an equivalent finite measure}
We now apply the above results to Borel actions preserving an infinite measure.
Given a measurable action $\mathfrak T$ of the group $G$, we denote by $\mathfrak M_{\inv}^\infty (\mathfrak T)$ the subset of $G$-invariant measures in $\mathfrak M^\infty$, by $\mathfrak M_{\erg}^\infty(\mathfrak T)$ the subset of $G$-invariant ergodic measures in $\mathfrak M^\infty$. It is not clear whether the sets $\mathfrak M_{\inv}^\infty (\mathfrak T)$ and $\mathfrak M_{\erg}^\infty (\mathfrak T)$ are Borel. It will be therefore convenient to consider smaller subsets of $\mathfrak M^\infty$, namely, of measures
that assign finite integral to a given positive measurable function.

 To simplify notation, consider the space $X$ fixed and omit it from notation, writing, for instance, $\mathfrak M$ instead of $\mathfrak M(X)$. Also, for a measure $\nu \in \mathfrak M^\infty$ and $f \in L_1 (X, \nu)$ write
\[
\nu(f) \;=\; \int f \, d\nu.
\]
Given a positive measurable function $f$ on $X$, we set
\[
\mathfrak M_f^\infty \;=\; \left\{ \nu \in \mathfrak M^\infty \, : \, f \in L_1 (X,\nu) \right\}.
\]
Introduce a map
\[
P_f: \, \mathfrak M^\infty_f \longrightarrow \mathfrak M
\]
by the formula
\begin{equation}
\label{defpf}
P_f (\nu) \; = \; \frac{f\nu}{\nu(f)}.
\end{equation}

Introduce a cocycle $\rho_f$ over the action $\mathfrak T$ by the formula
\[
\rho_f(g,x) \;=\; \frac{f(T_gx)}{f(x)}.
\]
A measure $\nu \in \mathfrak M_f^\infty$ is $\mathfrak T$-invariant if and only if
\[
P_f(\nu) \in \mathfrak M(\mathfrak T, \rho_f).
\]

Denote
\begin{align*}
\mathfrak M_{f,1}^\infty \;=\; \left\{ \nu \in \mathfrak M_f^\infty \, : \, \nu(f) = 1 \right\}; \\
\mathfrak M_{f,1,\inv}^\infty (\mathfrak T) \;=\; \mathfrak M_{f,1}^\infty \cap \mathfrak M_{\inv}^\infty (\mathfrak T); \\
\mathfrak M_{f,1,\erg}^\infty (\mathfrak T) \;=\; \mathfrak M_{f,1}^\infty \cap \mathfrak M_{\erg}^\infty (\mathfrak T).
\end{align*}
The set $\mathfrak M_{f,1}^\infty$ is Borel by definition. The map $P_f$ yields a Borel isomorphism of Borel spaces $\mathfrak M_{f,1}^\infty$ and $\mathfrak M$; the former is consequently a standard Borel space.
Furthermore, we clearly have
\begin{align*}
P_f (\mathfrak M_{f,1,\inv}^\infty ) &\;=\; \mathfrak M(\mathfrak T, \rho_f); \\
P_f (\mathfrak M_{f,1,\erg}^\infty ) &\;=\; \mathfrak M_{\erg}(\mathfrak T, \rho_f).
\end{align*}

In order to be able to apply Theorem \ref{ergdecstrcont} to
$\mathfrak M(\mathfrak T, \rho_f)$, we need an additional assumption on the function $f$.

{\bf{Definition.}} A Borel measurable function $f: X \to \mathbb R$ is said to be {\it fibrewise continuous} if for any $x \in X$ the function $f(T_g x)$ is continuous in $g \in G$.

In particular, if $X$ is a metric space, and the action ${\mathfrak T}$ is itself continuous, then
any continuous function is a fortiori fibrewise continuous. To produce continuous integrable functions, one can use the following simple proposition.
\begin{proposition}
Let $X$ be a metric space, and let $\nu$ be a sigma-finite Borel measure on $X$ assigning finite weight to every ball.
Then the space $L_1(X, \nu)$ contains a positive continuous function.
\end{proposition}
\begin{proof}
Let $d$ be the metric on $X$, take $x_0\in X$, let $\psi: {\mathbb R}_+\to {\mathbb R}_{>0}$ be positive, bounded and continuous, and set $f(x)=\psi(d(x,x_0))$. The mass of every ball is finite, so, if the function $\psi$ decays fast enough at infinity, then
$f\in L_1(X, \nu)$.
\end{proof}

If the function $f$ is fibrewise continuous then the cocycle $\rho_f$ given by the formula
\[
\rho_f (g,x) \;=\; \frac{f(T_g x)}{f(x)}
\]
is fibrewise continuous as well.
Consequently, the sets $M_{f,1,\inv}^\infty$ and $M_{f,1,\erg}^\infty$ are Borel subsets of $\mathfrak M^\infty$, and so are the sets $\mathfrak M_{f,\inv}^\infty$ and $\mathfrak M_{f,\erg}^\infty$.

Without losing generality assume $\nu(f)=1$ and consider the ergodic decomposition
\begin{equation}
\label{ergfnu}
f\nu=\displaystyle\int\limits_{\mergtrhof} \eta\, d\check \nu (\eta)
\end{equation}
of the measure $f\nu$ in ${\mathfrak M}({\mathfrak T}, \rho_f)$.
Dividing by $f$, we now obtain an ergodic decomposition
\begin{equation}
\label{nufone} \nu=\displaystyle\int\limits_{{\mathfrak M}^{\infty}_{f,1,erg}} \eta\, d\tilde \nu (\eta)
\end{equation}
of the initial measure $\nu$; note that, by construction, the correspondence $\nu\to{\tilde \nu}$ is bijective.

Theorem \ref{ergdecstrcont} now implies the following
\begin{corollary}
\label{infdec}
Let $\mathfrak T$ be a measurable action of an inductively compact group $G$ on a standard Borel space $(X, \mathcal B)$.
Let $f:X\to {\mathbb R}_{>0}$ be measurable, positive and fibrewise continuous. Then:
\begin{enumerate}
\item The sets $\mathfrak M_{f, 1,\inv}^\infty (\mathfrak T)$ and $\mathfrak M_{f, 1,\erg}^\infty (\mathfrak T)$ are Borel subsets of $\mathfrak M^\infty (X)$.
\item Every measure $\eta \in \mathfrak M_{f, 1,\erg}^\infty (\mathfrak T)$ is indecomposable in $\mathfrak M_{f, 1,\inv}^\infty (\mathfrak T)$.
\item For any $\nu \in \mathfrak M_{f, 1,\inv}^\infty (\mathfrak T)$ there exists a unique Borel probability measure $\overline \nu$ on $\mathfrak M_{f, 1,\erg}^\infty (\mathfrak T)$ such that
\begin{equation}
\label{eq-ergdecf}
\nu = \displaystyle\int\limits_{\mathfrak M_{f, 1,\erg}^\infty (\mathfrak T)} \eta \, d\overline \nu(\eta).
\end{equation}
The bijective correspondence $\nu \to \overline \nu$ is a Borel isomorphism of Borel spaces $\mathfrak M_{f, 1,\inv}^\infty (\mathfrak T)$ and $\mathfrak M \left( \mathfrak M_{f, 1,\erg}^\infty (\mathfrak T) \right)$.
\end{enumerate}
\end{corollary}

Corollary \ref{infdec} immediately implies 
\begin{corollary}
Let $\mathfrak T$ be a measurable action of an inductively compact group $G$ on a standard Borel space $(X, \mathcal B)$, and let $\nu$ be a $\sigma$-finite $\mathfrak T$-invariant Borel measure on $X$ such that the space $L_1(X, \nu)$ contains a positive
Borel measurable fibrewise continuous function. Then the measure $\nu$ admits an ergodic decomposition.
\end{corollary}
Indeed, an ergodic decomposition is obtained by  taking the positive
Borel measurable fibrewise continuous function $f\in L_1(X, \nu)$,
and dividing by $f$ the decomposition (\ref{eq-ergdecf}) of the measure $f\nu$.
Such an ergodic decomposition is of course not unique and 
depends on the choice of the positive Borel measurable fibrewise continuous integrable function.

It is convenient to allow more general ergodic decompositions of infinite measures. 
Given a measure $\nu \in \mathfrak M^\infty (X)$ and a $\sigma$-finite Borel measure $\overline \nu$ on $\mathfrak M^\infty (X)$, the equality
\begin{equation}
\label{ergdecminf}
\nu = \displaystyle\int\limits_{\mathfrak M^\infty(X)} \eta \, d \overline \nu(\eta)
\end{equation}
will always be understood in a similar way as above, in the following weak sense. Given a Borel set $A$,
as above we consider the function
\[
\myint_A: \mathfrak M^\infty \to \mathbb R_{\ge 0} \cup \left\{ \infty \right\}
\]
defined by
\[
\myint_A(\eta) \;=\; \eta(A).
\]

The equality (\ref{ergdecminf}) means that for any Borel set $A$ satisfying
$
\nu(A) < +\infty
$ we have
$\myint_A \in L_1(\mathfrak M^\infty (X), \overline \nu)$ and
\[
\nu(A) \;=\; \displaystyle\int\limits_{\mathfrak M^\infty (X)} \eta (A) \, d\overline \nu (\eta).
\]

For a measure $\nu$ invariant under the action $\mathfrak T$, a decomposition
\begin{equation}
\label{ergdecinf}
\nu \;=\; \displaystyle\int\limits_{\mathfrak M^\infty (X)} \eta \, d\overline\nu (\eta)
\end{equation}
will be called an ergodic decomposition of $\nu$ if $\overline\nu$ is a $\sigma$-finite measure on $\mathfrak M^\infty (X)$ and $\overline\nu$-almost all measures $\eta \in \mathfrak M^\infty (X)$ are invariant and ergodic with respect to the action $\mathfrak T$. Such decomposition is, of course, far from unique: indeed, if
\[
\varphi: \, \mathfrak M^\infty (X) \to \mathbb R
\]
is a Borel measurable function such that $\varphi(\eta) > 0$ for $\overline \nu$-almost all $\eta$, then a new decomposition is obtained by writing
\[
\nu \;=\; \displaystyle\int\limits_{\mathfrak M^\infty(X)} \frac{\eta}{\varphi(\eta)}\, d\left( \varphi(\eta) \overline \nu(\eta)\right).
\]
\subsubsection{Projective measures and admissibility}
As before, we consider the space $X$ fixed and omit it from notation.
Introduce the projective space ${\mathbb P}\mathfrak M^\infty$, the quotient of $\mathfrak M^\infty$
by the projective equivalence relation $\sim$ defined in the usual way:
\[
\nu_1 \sim \nu_2 \quad \text{if} \quad \nu_1 = \lambda \nu_2 \quad \text{for some } \lambda>0.
\]
Let
\[
{\bf p}: \mathfrak M^\infty \to {\mathbb P} \mathfrak M^\infty
\]
be the natural projection map. Elements of ${\mathbb P} \mathfrak M^\infty$
will be called projective  measures; finiteness, invariance, quasi-invariance and  ergodicity
of projective measures are defined in  the obvious way, and we denote
$$
{\mathbb P}\mathfrak M_{\inv}^\infty (\mathfrak T)= {\bf p}(\mathfrak M_{\inv}^\infty (\mathfrak T)); \
{\mathbb P}\mathfrak M_{\erg}^\infty (\mathfrak T)={\bf p}(\mathfrak M_{\erg}^\infty (\mathfrak T)).
$$

The Borel structure in the space ${\mathbb P}\mathfrak M^\infty$ is defined in the usual way: a set $A\subset {\mathbb P} \mathfrak M^\infty$ is Borel if its preimage ${\bf p}^{-1}(A)$ is Borel.

{\bf Definition.}
A measure $\overline \nu\in {\mathfrak M}^{\infty}({\mathfrak M}^{\infty})$  is called {\it admissible} if
the  projection map ${\bf p}$ is  $\overline \nu$-almost surely a bijection.

For example, any measure supported on the set ${\mathfrak M}^{\infty}({\mathfrak M})$ or, for a positive measurable $f$, on the set
${\mathfrak M}^{\infty}({\mathfrak M}^{\infty}_{f,1})$,
is automatically admissible.

If the measure $\overline \nu$ in an  ergodic decomposition (\ref{ergdecinf}) is admissible, then the ergodic decomposition is called admissible as well.

The following theorem shows that for a given invariant sigma-finite measure $\nu$,
the measure class of the measure ${\bf p}_* \overline\nu$ is the same for all admissible ergodic decompositions  (\ref{ergdecinf}).

\begin{theorem}
\label{ergdecinfpcl}
Let $\mathfrak T$ be a measurable action of an inductively compact group $G$ on a standard Borel space $(X, \mathcal B)$, and let $\nu$ be a $\sigma$-finite $\mathfrak T$-invariant Borel measure on $X$ such that the space $L_1(X, \nu)$ contains a positive
Borel measurable fibrewise continuous function. Then there exists a measure class $PCL(\nu)$ on ${\mathbb P} \mathfrak M^\infty$ with the following properties.
\begin{enumerate}
\item For any $\widetilde \nu \in PCL(\nu)$ we have $\widetilde\nu({\mathbb P}{\mathfrak M}^{\infty} \setminus
{\mathbb P}{\mathfrak M}^{\infty}_{\erg}({\mathfrak T}))=0$.
\item For any admissible ergodic decomposition
\[
\nu = \displaystyle\int\limits_{\mathfrak M^\infty} \eta\, d\overline\nu (\eta)
\]
of the measure $\nu$ we have ${\bf p}_*\overline \nu \in PCL(\nu)$.
\item Conversely, for any $\sigma$-finite Borel measure ${\widetilde \nu} \in PCL(\nu)$ there exists a unique
admissible $\sigma$-finite Borel measure $\overline \nu$ on $\mathfrak M^\infty(X)$ such that
${\bf p}_*{\overline \nu}={\widetilde \nu}$
and
\[
\nu \;=\; \displaystyle\int\limits_{\mathfrak M^\infty (X)} \eta \, d\overline\nu (\eta).
\]
\item Let $\nu_1$ and $\nu_2$ be two $\mathfrak T$-invariant $\sigma$-finite Borel measures, each admitting a positive fibrewise continuous integrable function. Then $\nu_1 \ll \nu_2$ if and only if $PCL(\nu_1) \ll PCL(\nu_2)$ and $\nu_1\perp \nu_2$
if and only if $PCL(\nu_1)\perp PCL(\nu_2)$.
In particular, $PCL(\nu_1) = PCL(\nu_2)$ if and only if $[\nu_1]=[\nu_2]$.
\end{enumerate}
\end{theorem}

\subsubsection{Infinite measures all whose ergodic components are finite}

Consider the set ${\mathbb P}{\mathfrak M}$ of {\it finite} projective measures and let $\nu$ be a sigma-finite invariant measure
such that $PCL(\nu)$ is supported on ${\mathbb P}{\mathfrak M}$. In this case take an arbitrary ergodic decomposition
\[
\nu = \displaystyle\int\limits_{\mathfrak M^\infty} \eta\, d\tilde\nu (\eta)
\]
and deform it by writing
\[
\nu = \displaystyle\int\limits_{\mathfrak M^\infty} \frac{\eta}{\eta(1)}\, \eta(1)d\tilde\nu (\eta).
\]

In this way we obtain an ergodic decomposition
$$
\nu = \displaystyle\int\limits_{\mathfrak M_{erg}({\mathfrak T})}{\eta}d\overline \nu (\eta),
$$
where the measure $\overline \nu$, supported on $\mathfrak M_{erg}({\mathfrak T})$, is uniquely defined
by $\nu$.

\noindent{\bf Acknowledgements.} Grigori Olshanski posed the problem to me; I am deeply grateful to him.
I am deeply grateful to Klaus Schmidt for kindly explaining to me the construction of Section 5 in \cite{GreSchm}
and for many very helpful discussions. I am deeply grateful to Yves  Coud{\`e}ne for  
kind explanations of Souslin theory. I am deeply grateful to Vadim Kaimanovich and Sevak Mkrtchyan
for useful discussions. I am deeply grateful to Nikita Kozin for typesetting parts of the manuscript.
Part of this work was done while I was visiting the Institut Mittag-Leffler in Djursholm,
the Erwin Schroedinger Institute in Vienna and the
Max Planck Institute in Bonn; I am deeply grateful to these institutions for their hospitality.

This work was supported in part by an Alfred P. Sloan Research Fellowship, by the
Dynasty Foundation Fellowship, by Grants MK-4893.2010.1 and MK-6734.2012.1 of
the President of the Russian Federation,
by the Programme on Dynamical Systems and Mathematical Control Theory of the
Presidium of the Russian Academy of Sciences,
by the RFBR-CNRS grant 10-01-93115, by the RFBR grant 11-01-00654 and 
by the Edgar Odell Lovett Fund at Rice University.

\section{Averaging operators.}
\subsection{Averaging over orbits of compact groups.}
Let $K$ be a compact group endowed with the Haar measure $\mu_K$ and let $\mathfrak T_K$ be a measurable action of $K$ on a standard Borel space $(X, \mathcal B)$. Let $\rho$ be a positive multiplicative real-valued measurable cocycle over the action $\mathfrak T_K$. Let $\mathbb B(X)$ be the space of bounded measurable functions on $X$ endowed with the Tchebychev metric. Introduce an operator $\mathcal A_K^\rho: \, \mathbb B(X) \to \mathbb B(X)$ by the formula
\begin{equation} \label{arhok}
\left( \mathcal A_K^\rho f\right) (x)=
\begin{cases}
\frac{\displaystyle \int\limits_K f(T_k x) \rho(k,x)\, d\mu_K (k)}{\displaystyle\int\limits_K \rho (k,x) \, d\mu_K(k)}
&\text{if $\displaystyle\int\limits_K \rho (k,x) \, d\mu_K(k)<+\infty$}\\
                  0,&\text{if $\displaystyle\int\limits_K \rho (k,x) \, d\mu_K(k)=+\infty$.}
\end{cases}
\end{equation}

It is clear that $\mathcal A_K^\rho$ is a positive contraction on the space $\mathbb B(X)$.

Let $\mathcal I_K$ be the $\sigma$-algebra of $K$-invariant subsets of $X$, and, for a given measure $\nu$, let $\mathcal I_K^\nu$ be the completion of $\mathcal I_K$ with respect to $\nu$.

As before, $\mathfrak M(\mathfrak T_K, \rho)$ stands for the space of Borel probability measures on $X$ with Radon-Nikodym cocycle $\rho$ with respect to the action $\mathfrak T_K$.

\begin{lemma}\label{avcomp}
 For any $\nu \in \mathfrak M(\mathfrak T_K, \rho)$  and any $f\in L_1(X, \nu)$ both integrals on the right-hand side of (\ref{arhok}) are
 $\nu$-almost surely finite. The extended operator $\mathcal A_K^\rho$ is a positive contraction of $L_1(X, \nu)$, and
 we have the $\nu$-almost sure equality
\begin{equation}
\label{condexpcomp}
\mathcal A_K^\rho f \;=\; \mathbb E(f \;\big|\; \mathcal I_K^\nu ).
\end{equation}
\end{lemma}
{\bf Remark.} Note that the left-hand side of (\ref{condexpcomp})
does not depend on the measure $\nu$, only on the cocycle $\rho$. This simple observation will
be important in what follows.

\begin{proof}
Let $\rho_x: K \to \mathbb R$ be defined by the formula
\[
\rho_x(k) = \rho(k,x).
\]
From the Fubini Theorem it immediately follows that for $\nu$-almost every $x \in X$ we have $\rho_x \in L_1(K,\mu_K)$.
Now take $\varphi \in L_1(X,\nu)$ and set
\[
\varphi_x (k) = \varphi(T_k x) \rho(k,x).
\]

\begin{proposition}
For $\nu$-almost every $x \in X$ we have $\varphi_x \in L_1 (K,\mu_K)$.
\end{proposition}

 Consider the product space $K\times X$ endowed with the measure $\tilde \nu$ defined by the formula
\begin{equation}
\label{nutilde}
d\tilde \nu = \rho (k,x) \, d\mu_K \, d\nu.
\end{equation}

For any fixed $k_0 \in K$ we have
\[
\displaystyle\int_X \rho(k_0, x) \, d\nu(x) = 1,
\]
whence $\tilde \nu$ is a probability measure.

For any $k \in K$ we have
\[
\displaystyle\int |\varphi(T_k x) | \cdot \rho(k,x) \, d\nu(x) = \displaystyle\int | \varphi(x) | \, d\nu(x),
\]
so the function $\tilde \varphi(k,x) = \varphi(T_k x)$ satisfies $\tilde\varphi \in L_1 (K \times X, \tilde \nu)$.
The claim of the Proposition follows now from the Fubini Theorem.

We return to the proof of Lemma \ref{avcomp}.
First, the cocycle property implies that
\[
{\mathcal A}_K^\rho \varphi (x) = {\mathcal A}_K^\rho \varphi (T_k x)
\]
for any $k \in K$.
By the Fubini Theorem applied to the space $K \times X$ endowed with the measure $\tilde\nu$, for any Borel subset $A \subset X$ and any $\tilde \varphi \in L_1(K \times X, \tilde \nu)$ we have:
\begin{multline}
\displaystyle\int\limits_A\displaystyle\int\limits_K
\tilde\varphi(k,x) \, \rho(k,x) \, d\mu_K (k) \, d\nu(x) =\\
=\displaystyle\int\limits_A\displaystyle\int\limits_K
\left(
\frac{\displaystyle\int\limits_K \tilde\varphi(k,x) \, \rho(k,x) \, d\mu_K (k)}{\displaystyle\int\limits_K
\rho(k,x) \, d\mu_K (k) }
\right)
\rho(k,x) \, d\mu_K (k)\, d\nu(x).
\end{multline}
Now take $\varphi \in L_1(X,\nu)$ and apply the above formula to the function
\[
\tilde\varphi(k,x) = \varphi(T_kx)
\]
(note here that $\tilde\varphi \in L_1(K\times X, \tilde\nu)$ by Fubini's theorem). We obtain
\[
\displaystyle\int\limits_K \left( \, \displaystyle\int\limits_A \varphi(T_k x) \, d\nu \circ T_k(x) \right) \, d\mu_K(k) =
\displaystyle\int\limits_K \left( \, \displaystyle\int\limits_A {\mathcal A}_K^\rho \varphi(x) \, d\nu \circ T_k (x) \right) \, d\mu_K (k).
\]

Now let the set $A$ be $K$-invariant. Recalling that the function ${\mathcal A}_K^\rho \varphi$ is $K$-invariant as well, we finally obtain
\[
\displaystyle\int\limits_A \varphi(x) \, d\nu(x) = \displaystyle\int_A {\mathcal A}_K^\rho \varphi(x) \, d\nu(x),
\]
and the Lemma is proved completely.
\end{proof}

 \subsection{Averaging over orbits of inductively compact groups.}
As above, let
\[
G = \bigcup_{n=1}^{+\infty} K(n), \ K(n)\subset K(n+1)
\]
be an inductively compact group, and let $\mu_{K(n)}$ denote the Haar measure on the group $K(n)$.
Assume we are given a measurable  action ${\mathfrak T}$ of $G$
on a standard Borel space $(X, \mathcal{B})$.
Let $\mathcal{I}_{K(n)}$ stand for the $\sigma$-algebra of $K(n)$ -- invariant measurable subsets of X, and  let $\mathcal{I}_G$ be the $\sigma$-algebra of $G$-invariant subsets of $X$. Clearly, we have
\[
\mathcal{I}_G = \bigcap_{n=1}^{\infty} \mathcal{I}_{K(n)}.
\]

Let $\rho$ be a positive measurable multiplicative cocycle over the action ${\mathfrak T}$.

The  averaging operators $\mathcal{A}_{K(n)}^\rho$, $n\in {\mathbb N}$, are defined,
for a bounded measurable function $\varphi$ on $X$, by formula (\ref{arhok}).
For brevity, we shall sometimes write $\mathcal{A}_{n}^\rho=\mathcal{A}_{K(n)}^\rho$.

Now take $\nu\in {\mathfrak M}({\mathfrak T}, \rho)$ and let
$\mathcal{I}^{\nu}_{K(n)}$, $\mathcal{I}^{\nu}_{G}$ be the completions of the sigma-algebras $\mathcal{I}_{K(n)}$, $\mathcal{I}_{G}$ with respect to the measure $\nu$.

By the results of the previous subsection, for any $\varphi\in L_1(X, \nu)$, we have the $\nu$-almost sure equality
\[
\mathcal{A}_n^\rho \varphi = \mathbb{E} (\varphi \, \big| \, \mathcal{I}^{\nu}_{K(n)} ).
\]

Since $\mathcal{I}^{\nu}_{K(n+1)} \subset \mathcal{I}^{\nu}_{K(n)}$, the reverse martingale convergence theorem implies the following
\begin{proposition} For any $\varphi \in L_1(X,\nu)$ we have
\[
\lim\limits_{n\to\infty}\mathcal{A}_n^\rho \varphi=\mathbb{E} (\varphi \, \big| \, \mathcal{I}^{\nu}_G )
\]
both $\nu$-almost surely and in $L_1(X,\nu)$.
\end{proposition}

Introduce the averaging operator $\mathcal{A}_\infty^\rho$ by setting
\[
\mathcal{A}_\infty^\rho \varphi(x) = \lim\limits_{n \to \infty} \mathcal{A}_n^\rho \varphi(x).
\]
If for a given $x\in X$  the sequence $\mathcal{A}_n^\rho \varphi(x)$ fails to converge, then the value $\mathcal{A}_\infty^\rho \varphi(x)$ is not defined.
From the definitions and the Reverse Martingale Convergence Theorem we immediately have
\begin{proposition}
A measure $\eta\in {\mathfrak M}({\mathfrak T}, \rho)$ is ergodic of and only if for
any $\varphi \in L_1(X, \eta)$ we have $$\mathcal{A}_\infty^\rho \varphi(x) = \displaystyle\int_X \varphi \, d\eta$$
almost everywhere with respect to the measure $\eta$.
\end{proposition}

Conversely, we  have
\begin{proposition}
\label{denserg}
Let $\eta\in {\mathfrak M}({\mathfrak T}, \rho)$  and assume that there exists a
dense set $\Psi \subset L_1(X,\eta)$ such that for any $\psi \in \Psi$ we have
\[
\mathcal{A}_\infty^\rho \psi = \displaystyle\int \psi \, d\eta
\]
almost surely with respect to $\eta$. Then the measure $\eta$ is ergodic.
\end{proposition}

\subsection{Equivalence of indecomposability and ergodicity: proof of Proposition \ref{equivergindec}.}

\begin{proposition}  Let $A$ be a $G$-almost-invariant Borel subset of $X$.
Then there exists a $G$-invariant Borel set $\tilde A$ such that
\[
\nu (A \vartriangle \tilde A) = 0.
\]
\end{proposition}
\begin{proof}
Let $\chi_A$ be, as usual,  the indicator function of $A$.
If $A$ is  $G$-almost-invariant, then for almost every $x \in A$ and all $n \in \mathbb{N}$ we have
\[
\mathcal{A}_n^{\rho} \chi_A (x) = 1.
\]
Indeed, consider the set $K(n) \times A$ endowed with the product measure $\mu_{K(n)} \times \nu$. For almost all points $(k,x) \in K(n) \times A$ by definition we have $T_k x \in A$. By Fubini's theorem, for almost every $x \in X$ the set $\{ k \in K(n): T_k x \in A\}$ has full measure, whence $\mathcal{A}_n^\rho \chi_A (x) = 1$ as desired.

Now introduce the set $\tilde A$ as follows:
\[
\tilde A = \{ x \in X: \mathcal{A}_n^\rho \chi_A(x) = 1 \text{ for all sufficiently large } n \in \mathbb{N} \}.
\]
By definition, $\tilde{A} \supset A$. On the other hand, since for $x \in \tilde{A}$ we have $\mathcal{A}_\infty^\rho \chi_A (x) = 1,$
the equality
\[
\displaystyle\int\limits_X \mathcal{A}_\infty^\rho \chi_A \, d\nu = \nu(A)
\]
implies $\nu(\tilde A) \le \nu(A)$, whence $\nu (\tilde A \vartriangle A) = 0$ and the proposition is proved.
\end{proof}

\subsection{The set of ergodic measures is Borel.}

\begin{proposition}
\label{ergborel}
Let $\rho$ be a fibrewise continuous cocycle
over a measurable action $\mathfrak T$ of an inductively compact group $G$ on a standard Borel space $(X, {\mathcal B})$.
Then the  set $\mathfrak M_{\erg} (\mathfrak T, \rho)$ is a Borel subset of $\mathfrak M(X)$.
\end{proposition}

\begin{proof}

We start with the following auxiliary proposition.

\begin{proposition}
\label{setphi} Let $(X, {\mathcal B})$ be a standard Borel space.
 There exists a countable set $\Phi$ of bounded measurable functions on $X$ such that for any probability measure $\nu$ on $X$ and any bounded measurable function $\varphi: X \to \mathbb{R}$ there exists a sequence $\varphi_n \in \Phi$ such that
\begin{enumerate}
\item $\displaystyle \sup_{n \in \mathbb{N}, x \in X} \varphi_n(x) < +\infty$;

\item  $\varphi \to \varphi$ as $n \to \infty$ almost surely with respect to $\nu$.
\end{enumerate}
\end{proposition}

\begin{proof} On the unit interval  take the family of piecewise-linear functions with nodes at rational points.
\end{proof}

We return to the proof of Proposition \ref{ergborel}.
It is clear that for any fixed bounded measurable function $\varphi$ on $X$ the set
\[
\{ \nu : \lim_{n \to \infty} \mathcal A_n^\rho \varphi \text{ exists and is constant } \nu\text{-almost surely} \}
\]
is Borel. Intersecting over all $\varphi \in \Phi$ and using Proposition \ref{denserg}, we obtain the claim.
\end{proof}

\section{The sigma-algebra of $G$-invariant sets.}
\subsection{Measurable partitions in the sense of Rohlin.}\subsubsection{Lebesgue spaces}
A triple $(X, {\mathcal B}, \nu)$, where $X$ is a set, ${\mathcal B}$ a sigma-algebra on $X$, and $\nu$ a measure on $X$, defined on
${\mathcal B}$ and such that ${\mathcal B}$ is complete with respect to $\nu$ is called a {\it Lebesgue space} if it is either countable or measurably isomorphic to the unit interval endowed with the sigma-algebra of Lebesgue measurable sets and the Lebesgue measure (perhaps with
a countable family of atoms). No Borel structure on $X$ is assumed in this definition.
\subsubsection{Measurable partitions}
A partition $\xi$ of $X$ is simply a representation of $X$ as a disjoint union of measurable sets:
$$
X=\bigcup X_{\alpha}.
$$
The sets $X_{\alpha}$ are called {\it elements} of the partition. For a point $x$, the element of the partition $\xi$ containing $x$ will be denoted ${\mathcal C}_{\xi}(x)$.
A family of sets $\Psi$ is said to be {\it a basis} for the
partition $\xi$ if for any two elements $X_1, X_2$ of $\xi$ there exists a set $A_1$ in $\Psi$
containing $A_1$ and disjoint from $A_2$.
A {\it measurable partition} $\xi$  of $(X, \mathcal B, \nu)$ is by definition a partition
of a subset $Y\subset X$ of full measure which admits a countable basis.

Following Rohlin, to a measurable partition $\xi$ we assign the quotient space ${\overline X}(\xi)$ whose points
are elements of the partition $\xi$. We have a natural almost surely defined projection map $\pi_{\xi}: X\to {\overline X}(\xi)$, which endows the  set ${\overline X}(\xi)$ with a natural sigma-algebra ${\overline {\mathcal B}}(\xi)$, the push-forward of ${\mathcal B}$,
and the natural quotient-measure ${\overline \nu}_{\xi}$, the push-forward of the measure $\nu$.
Rohlin proved that the space  $({\overline X}(\xi), {\overline {\mathcal B}}(\xi), {\overline \nu}_{\xi})$ is again
a Lebesgue space. Furthermore, Rohlin showed that the measure $\nu$ admits the {\it canonical system of conditional measures} defined as follows.  For ${\overline \nu}_{\xi}$-almost every element ${\mathcal C}$ of the partition $\xi$ there is  a
probability measure $\nu_{{\mathcal C}}$ on ${\mathcal C}$ such that for any set $A\in {\mathcal B}$ the function
${\rm int}_A: {\overline X}(\xi)\to {\mathbb R}$ given by the formula ${\rm int}_A({\mathcal C})=\nu_{{\mathcal C}}(A)$ is
${\overline {\mathcal B}}$-measurable and we have
\begin{equation}
\nu(A)=\displaystyle\displaystyle\int\limits_{{\overline X}(\xi)} \nu_{{\mathcal C}}(A) d{\overline \nu}_{\xi}({\mathcal C}).
\end{equation}
This system of canonical conditional measures is unique: any two systems coincide ${\overline \nu}_{\xi}$-almost surely.
To a measurable partition $\xi$ we now assign an averaging operator ${\mathcal A}_{\xi}$ on $L_1(X, \nu)$, given by the formula
\begin{equation}
{\mathcal A}_{\xi}f(x)=\int_{{\mathcal C}_{\xi}(x)} f(x) d\nu_{{\mathcal C}_{\xi}(x)}
\end{equation}
(the right-hand side is defined $\nu$-almost surely by Rohlin's Theorem).
Given a measurable partition $\xi$, let ${\mathcal B}_{\xi}$ be the sigma-algebra of measurable subsets of $X$
which are unions of elements of $\xi$ and a set of measure zero.
Rohlin proved that for any    $f\in L_1(X, \nu)$ we have the $\nu$-almost sure identity
\begin{equation}
{\mathbb E}(f| {\mathcal B}_{\xi})={\mathcal A}_{\xi}f.
\end{equation}

Rohlin showed, furthermore, that every complete sub-sigma-algebra ${\mathcal B}_1\subset {\mathcal B}$ has the form
${\mathcal B}_1={\mathcal B}_{\xi}$ for some measurable partition $\xi$ of the Lebesgue space $(X, {\mathcal B}, \nu)$.

\subsection{Borel partitions}
Let $(X, \mathcal B)$ be a standard Borel space. A decomposition
\[
X \;=\; \bigsqcup\limits_{\alpha} X_\alpha,
\]
where $\alpha$ takes values in an arbitrary index set and where, for each $\alpha$, the set $X_\alpha$ is Borel, will be called a {\it Borel partition} if there exists a countable family
\[
Z_1, \ldots, Z_n, \ldots
\]
of Borel sets such that for any two indices $\alpha_1, \alpha_2$ where $\alpha_1 \ne \alpha_2$, there exists $i \in \mathbb N$ satisfying
\[
X_{\alpha_1} \subset Z_i, \qquad X_{\alpha_2} \cap Z_i = \emptyset.
\]
In this case, the countable family will be called the {\it countable basis} for the partition.

If $\nu$ is a Borel probability measure on $X$, then the space $(X, {\mathcal B}, \nu)$ is a Lebesgue space in the sense of Rohlin, while a Borel partition now becomes a measurable partition in the sense of Rohlin. Observe that all conditional measures are in this case defined on the Borel sigma-algebra.

\subsection{The measurable partition corresponding to the sigma-algebra of invariant sets}
Our first aim is to give an explicit description of the measurable partition
corresponding to the $\sigma$-algebra $\mathcal{I}_G$ of $G$-invariant sets.

Let $\Phi$ be the set given by Proposition \ref{setphi} and write  $\Phi = \{ \varphi_1, \varphi_2, \ldots, \varphi_n, \ldots \}$.

Introduce a set $X(\Phi, \rho)$ by the formula:
\[
X(\Phi, \rho) \;=\; \{ x \in X : \, A_\infty^\rho \varphi_k(x) \text{ is defined for all } k \in \mathbb N\}.
\]
The set $X(\Phi, \rho)$ is clearly Borel. Observe that for any $\nu \in \mathfrak M(\mathfrak T, \rho)$ we have
\[
\nu(\mathfrak M(\mathfrak T, \rho)) = 1.
\]

Let $\mathbb R^{\mathbb N}$ be the space of all real sequences:
\[
\mathbb R^{\mathbb N} \;=\; \{ \mathbf r = (r_k), \, k \in \mathbb N,\, r_k \in \mathbb R \}.
\]
We endow $\mathbb R^{\mathbb N}$ with the usual product $\sigma$-algebra, which turns it into a standard Borel space. For $\mathbf r \in \mathbb R^{\mathbb N}$, we introduce a subset $X( \mathbf r, \Phi, \rho)$ by the formula
\[
X (\mathbf r, \Phi, \rho) \;=\; \{ x \in X(\Phi,\rho): \, A_\infty^\rho \varphi_k(x) = r_k, \, k \in \mathbb N \}.
\]
For any $\mathbf r \in \mathbb R^{\mathbb N}$, the set $X(\mathbf r, \Phi, \rho)$ is Borel, and we clearly have
\[
X(\Phi,\rho) \;=\; \bigsqcup\limits_{\mathbf r \in \mathbb R^{\mathbb N}}^{\quad} X( \mathbf r, \Phi, \rho).
 \]
It is clear from the definitions that the Borel partition
\[
X \;=\; (X \setminus X(\Phi, \rho)) \bigsqcup\bigsqcup\limits_{\mathbf r \in \mathbb R^{\mathbb N}} X(\mathbf r, \Phi, \rho)
\]
has a countable basis.

Introduce a map
\[
\Pi_\Phi: \, X(\Phi, \rho) \longrightarrow \mathbb R^{\mathbb N}
\]
by the formula
\[
\Pi_\Phi(x) \;=\; \left( \mathcal A^\rho_\infty \varphi_1(x), \,\ldots,\, \mathcal A^\rho_\infty \varphi_n(x),\, \ldots \right).
\]
The map $\Pi_\Phi$ is, by definition, Borel. Now, introduce a map
\[
\Int_\Phi: \, \mathfrak M(\mathfrak T, \rho) \longrightarrow \mathbb R^{\mathbb N}
\]
by the formula
\[
\Int_\Phi (\nu) \;=\; \left( \;\int_X \varphi_1\, d\nu, \, \ldots, \, \int_X \varphi_n\, d\nu, \, \ldots \;\right).
\]
The map $\Int\limits_\Phi$ is, by definition, Borel and injective.

By Souslin's Theorem (see \cite{souslin}, \cite{bogachev}, \cite{coudene}), 
it follows the sets $\Int_\Phi(\mathfrak M(\mathfrak T, \rho) )$ and $\Int_\Phi (\mathfrak M_{\erg} ( \mathfrak T, \rho))$ 
are Borel. Introduce a subset $X_{\erg} \subset X$ by the formula
\[
X_{\erg} \;=\; \Pi_\Phi^{-1} \left( \Int_\Phi \left( \mathfrak M_{\erg} (\mathfrak T, \rho) \right) \right).
\]
Again, Souslin's Theorem implies that the set $X_{\erg}$ is Borel.
We thus have the following diagram, all whose arrows correspond to Borel maps 
$$
\xymatrixcolsep{5pc}\xymatrix{
X_{\erg} \ar[d]^{\Pi_{\Phi}} \ar[rd]^{({\Int_{\Phi}})^{-1}\circ\Pi_{\Phi}} \\
{\mathbb R}^{\mathbb N}&
{\mathfrak M}_{\erg}(\rho, \mathfrak T) \ar[l]^{\Int_{\Phi}}
}
$$

We shall now see that for any $\nu \in \mathfrak M(\mathfrak T, \rho)$ we have
\[
\nu(X_{\erg}) \;=\;1.
\]
Indeed, take an arbitrary $\nu \in \mathfrak M(\mathfrak T, \rho)$. The Borel partition $\xi$ now induces a measurable partition that we denote $\xi^\nu$. Let $\overline X(\xi^\nu)$ be the space of elements of the partition $\xi$, or, in other words, the quotient of the space $X$ by the partition $\xi$. Let
\[
\pi_{\xi^\nu}: \,  X \longrightarrow \overline X(\xi^\nu)
\]
be the natural projection map, and let
\[
\widetilde \nu \;=\; (\pi_{\xi^\nu})_* \,\nu
\]
be the quotient measure on $\overline X(\xi^\nu)$.

By Rohlin's Theorem, $\widetilde \nu$-almost every element $\mathcal C$ of the partition $\xi^\nu$ carries a canonical conditional measure $\nu_{\mathcal C}$. The key step in the construction of the ergodic decomposition is given by the following Proposition.

\begin{proposition}\label{ergcond}
The measurable partition $\xi^\nu$ generates the $\sigma$-algebra $\mathcal I_G^\nu$, the $\nu$-completion of the $\sigma$-algebra of Borel $G$-invariant sets.
For $\widetilde \nu$-almost every $\mathcal C$ we have $\nu_\mathcal C\in\mathfrak M_{\erg} (\rho, \mathfrak T)$.
\end{proposition}

The Proposition will be proved in the following subsection. Rohlin's decomposition
\[
\nu \;=\; \int\limits_{\overline X(\xi^\nu)} \nu_{\mathcal C} \, d\widetilde \nu (\mathcal C)
\]
will now be used to obatain an ergodic decomposition of the measure $\nu$.

 Indeed, let the map
\[
\mes_{\xi^{\nu}}: \overline X(\xi^\nu) \longrightarrow \mathfrak M_{\erg} (\rho, \mathfrak T)
\]
be given by the formula
\[
\mes_{\xi^\nu} (\mathcal C) \;=\; \nu_{\mathcal C}.
\]

\begin{proposition}
The map $\mes_{\xi^\nu}$ is ${\widetilde \nu}$-measurable.
\end{proposition}
\begin{proof} Let $\varphi$ be a bounded measurable function on $X$. Let $\alpha \in \mathbb R$. By definition of the measurable structure on the quotient space ${\overline X}(\xi^{\nu})$,
 it suffices to show that the set
\[
\{ x \in X: \displaystyle\int \varphi \, d\nu_{\mathcal C(x)} > \alpha \}
\]
is $\nu$-measurable. But by Proposition \ref{ergcond} we have the $\nu$-almost sure equality
\[
\{ x \in X: \displaystyle\int \varphi \, d\nu_{\mathcal C(x)} > \alpha \} =
\{ x \in X: \mathcal A_\infty^\rho \varphi(x) > \alpha \}.
\]
 Since the set $\{ x \in X: \mathcal A_\infty^\rho \varphi(x) > \alpha \}$ is Borel, the Proposition is proved.
\end{proof}

For $x \in X$ let ${\mathcal C}_{\xi}(x)$ be the element of the partition $\xi$ containing $x$, and introduce a map
\[
\Mes_{\xi^{\nu}}: \, X \longrightarrow \mathfrak M_{\erg}(\rho, \mathfrak T)
\]
by the formula
\[
\Mes_{\xi^{\nu}} (x) \;=\; \nu_{{\mathcal C}_{\xi}(x)}.
\]

We have a commutative diagram

$$
\xymatrixcolsep{5pc}\xymatrix{
X \ar[d]^{\pi_{\xi^{\nu}}} \ar[rd]^{\Mes_{{\xi}^{\nu}}} \\
\overline X(\xi^\nu) \ar[r]^{\mes_{\xi^\nu}} &
{\mathfrak M}_{\erg}(\rho, \mathfrak T)
}
$$
In particular, the map $\Mes_{\xi^{\nu}}$ is $\nu$-measurable.
Proposition \ref{ergcond} immediately implies the following
\begin{corollary}
\label{mesxi}
For any $\nu\in M_{\erg} (\mathfrak T, \rho)$ we have $\nu(X_{\erg})=1$.
The equality
$$
\Mes_{\xi^{\nu}}=({\Int_{\Phi}})^{-1}\circ\Pi_{\Phi}
$$
holds $\nu$-almost surely.
\end{corollary}

Denoting
\[
\overline \nu \;=\; (\Mes_{\xi^\nu})_*\, \nu \;=\; (\mes_{\xi^\nu})_*\, \widetilde \nu,
\]
we finally obtain an ergodic decomposition
\[
\nu \;=\; \int\limits_{\mathfrak M_{\erg}(\rho, \mathfrak T)} \eta\, d\overline \nu (\eta)
\]
for the measure $\nu$.
To complete the proof of the first two claims of Theorem \ref{ergdecstrcont} it remains to establish Proposition \ref{ergcond}.

\subsection{Proof of Proposition \ref{ergcond}.}

\subsubsection{Proof of the first claim.}

\begin{proof} On one hand, every element of the partition $\xi^{\nu}$ is by definition $G$-invariant.

Conversely, let $A$ be $G$-invariant. Our aim is to find a measurable set $A'$ which is a union of elements of the partition $\xi^{\nu}$ and satisfies
\[
\nu(A \triangle A') = 0.
\]
Take a sequence $\varphi_{n_k} \in \Phi$ such that
\[
\sup_{k \in \mathbb N, \, x \in X} \varphi_{n_k}(x) < +\infty
\]
and $\varphi_{n_k} \to \chi_A$ almost surely with respect to the measure $\nu$ as $k \to \infty$.

Now let
\[
R_A = \{ {\bf r}\in {\mathbb R}^{\mathbb N}, {\bf r}=(r_n),  \lim\limits_{k \to \infty} r_{n_k} = 1 \}
\]
and let
\[
A' = \bigcup_{{\bf r} \in R_A} X(\rho, \Phi, {\bf r}),
\]
\[
A'' = \{ x \in X: {\mathcal A}_\infty^\rho \chi_A (x) = 1 \}.
\]
Since $A$ is $G$-invariant, we have
\[
\nu(A \triangle A'') = 0.
\]
Since $$
\lim\limits_{k\to\infty}\varphi_{n_k}=\chi_A
$$
$\nu$-almost surely
and all functions are uniformly bounded, we have
\[
{\mathcal A}_\infty^\rho \, \varphi_{n_k} \to {\mathcal A}_\infty^\rho \chi_A
\]
almost surely as $k \to \infty$. It follows that
\[
\nu(A' \triangle A'') = 0,
\]
and, finally, we obtain
\[
\nu (A \triangle A') = 0,
\]
which is what we had to prove.

\end{proof}

\subsubsection{Proof of the second claim.}

\begin{proposition} For every $g \in G$, for $\bar{\nu}$-almost every $\mathcal{C} \in {\overline X}(\xi^{\nu})$ and $\nu_\mathcal{C}$-almost every $x \in X$ we have
\[
\frac{d\nu_{\mathcal{C}} \circ T_g}{d\nu_\mathcal{C}}(x) = \rho(g,x).
\]
\end{proposition}

\begin{proof} This is immediate from the uniqueness of the canonical system of conditional measures. Indeed, on the one hand, we have
\[
\nu \circ T_g = \displaystyle\int\limits_{{\overline X}(\xi^{\nu})} \nu_\mathcal{C} \circ T_g \, d\bar{\nu}(\mathcal{C});
\]
on the other hand,
\[
\nu \circ T_g = \rho(g,x) \cdot \nu = \displaystyle\int\limits_{{\overline X}(\xi^{\nu})} \rho(g,x) \cdot \nu_{\mathcal{C}} \, d\bar{\nu} (\mathcal{C}),
\]
whence $\nu_\mathcal{C} \cdot T_g = \rho(g,x) \cdot \nu_\mathcal{C}$ for $\widetilde \nu$-almost all $\mathcal{C} \in {\overline X}(\xi^{\nu})$,
 and the Proposition is proved.
\end{proof}

Fibrewise continuity of the cocycle is necessary to pass from a countable dense subgroup to the whole group.

\begin{proposition} Let $\rho$ be a positive Borel fibrewise continuous cocycle over a measurable action ${\mathfrak T}_K$
of a compact group $K$ on a standard Borel space $(X,\mathcal B)$. Let $\nu$ be a Borel probability measure on $X$.
Let $K' \subset K$ be dense, and assume that the equality
\begin{equation} \label{eq:myeq01}
\frac{d\nu \circ T_k}{d\nu} = \rho(k,x)
\end{equation}
holds for all $k \in K'$. Then $\nu\in {\mathfrak M} ({\mathfrak T}_K, \rho)$.
\end{proposition}

\begin{proof} We start by recalling the following Theorem of Varadarajan (Theorem 3.2 in \cite{Var}).
\begin{theorem}[Varadarajan]
\label{varad}
Assume that a locally compact second countable group $K$ acts measurably on a standard Borel space $(X,\mathcal B)$.
There exists a compact metric space $Z$, a continuous action of $K$ on $Z$
and a $K$-invariant Borel subset $Z^{\prime}\subset Z$ such that the restricted action of $K$ on $Z^{\prime}$ is measurably
isomorphic to the action of $K$ on $(X,\mathcal B)$.
\end{theorem}

{\bf Question.} Under what assumptions does the same conclusion hold for Borel actions of inductively compact groups?

 We apply Varadarajan's Theorem to the action of our compact group $K$.
 Passing, if necessary, to the larger space given by the theorem,
 we may  assume that $X$ is a compact metric space, $\nu$ a Borel probability measure, and that the action of $K$ on $X$ is continuous. Consequently, if $k_n \to k_\infty$ in $K$ as $n \to \infty$, then
\[
\nu \circ T_{k_n} \to \nu \circ T_{k_\infty}
\]
weakly in the space of Borel probability measures on $X$. It remains to show that the measures $\nu = \rho(k_n,x) \cdot \nu$ weakly converge to the measure $\rho(k,x)\cdot \nu$ as $n \to \infty$, and the equality $\nu \circ T_{k_\infty} = \rho(k_\infty, x) \cdot \nu$ will be established. First of all, observe that the function
\[
\rho_{\max}(x) = \max_{k \in K} \rho(k,x)
\]
is well-defined and measurable in $X$ (since, by continuity, the maximum can be replaced by the supremum over a countable dense set). We shall show that for any bounded measurable function $\psi$ on $X$ we have
\[
\lim_{n \to \infty} \displaystyle\int\limits_X \psi (x) \rho(k_n,x) \, d\nu(x) = \displaystyle\int\limits_X \psi(x) \rho(k_\infty,x) \, d\nu(x).
\]
Assume $\psi$ satisfies $0 \le \psi \le 1$. For every $x \in X$ we have
\[
\lim_{n \to \infty} \rho(k_n,x) = \rho(k_\infty,x).
\]
By Fatou's Lemma,
\[
\displaystyle\int \psi (x) \rho(k_\infty,x) \, d\nu (x) \le \lim_{n \to \infty} \inf \displaystyle\int \psi (x) \rho(k_n,x) \, d\nu(x).
\]
For $N>0$ set $X_N = \{ x : \rho_{\max} (x) \le N \}$. Take $\varepsilon > 0$ and choose $N$ large enough in such a way that we have
\[
\nu(X\setminus X_N) < \varepsilon, \quad \displaystyle\int\limits_{X\setminus X_N} \psi (x) \rho(k_\infty,x) \, d\nu (x) < \varepsilon.
\]
Observe that since $X_n$ is $K$-invariant, for all $n \in \mathbb{N}$ we have
\[
\displaystyle\int\limits_{X\setminus X_N} \psi(x) \rho(k_n,x) \, d\nu(x) \le \nu \circ T_{k_n} (X\setminus X_N) = \nu(X \setminus X_N) < \varepsilon.
\]
By the bounded convergence theorem, we have
\[
\lim_{n \to \infty} \displaystyle\int\limits_{X_N} \psi (x) \rho(k_n,x) \, d\nu(x) = \displaystyle\int\limits_{X_N} \psi(x) \rho(k_\infty,x) \, d\nu(x),
\]
whence
\[
\displaystyle\int\limits_X \psi(x) \rho(k_\infty,x) \, d\nu(x) \ge \lim_{n \to \infty} \sup \displaystyle\int \psi (x) \rho(k_n,x)\, d\nu(x) -3\varepsilon.
\]
Since $\varepsilon$ is arbitrary, the proposition is proved.

We return to the
proof of the second claim of Proposition \ref{ergcond}.

 First, take $n_0 \in \mathbb{N}$ and show that for $\widetilde \nu$-almost every $\mathcal C$ and all $k \in K(n_0)$ we have
\begin{equation} \label{eq:myeq02}
\frac{d \nu_{\mathcal C} \circ T_k}{d \nu_{\mathcal C}} = \rho(k,x).
\end{equation}

Choose a countable dense subgroup $K' \subset K(n_0)$. The equality \eqref{eq:myeq02} holds for all $k \in K'$ and for $\widetilde \nu$-almost all $\mathcal C$. But then fibrewise continuity of the cocycle $\rho$ implies that \eqref{eq:myeq02} holds also for all $k \in K(n_0)$.  Consequently, $\nu_{\mathcal C}\in {\mathfrak M}(\rho, {\mathfrak T})$ for $\widetilde \nu$-almost all $\mathcal C$.
Now, by definition of the partition $\xi$,  for every $\varphi \in \Phi$ we have
\[
\mathcal A_\infty^\rho = \displaystyle\int \varphi \, d\nu_\mathcal C
\]
almost surely with respect to $\nu_\mathcal C$ (indeed, the function $\mathcal A_\infty^\rho \varphi$ is almost surely constant in restriction to $\mathcal C$, but then the constant must be equal to the average value).

Since $\Phi$ is dense in $L_1 (X, \nu_\mathcal C)$, and $\nu_\mathcal C\in {\mathfrak M}({\mathfrak T}, \rho)$, we conclude that $\nu_\mathcal{C}$ is ergodic for $\widetilde \nu$-almost every $\mathcal C$, and the Proposition is proved completely.

\end{proof}
\subsection{Uniqueness of the ergodic decomposition.}

Consider the map $$\Mes: \mtrho\to {\mathfrak M}(\mergtrho)$$ that
to a measure $\nu\in \mtrho$ assigns the measure
$$
{\overline \nu}= \Mes(\nu)=\left(\Mes_{\xi^{\nu}}\right)_*\nu.
$$
By definition, we have
\begin{equation}
\label{ergdec7}
\nu=\displaystyle\int\limits_{\mergtrho} \eta d{\overline \nu}(\eta).
\end{equation}
Conversely, introduce a map $ED: {\mathfrak M}(\mergtrho)\to \mtrho$ which takes a measure ${\overline \nu}\in {\mathfrak M}(\mergtrho)$ to the measure $\nu$ given by the formula (\ref{ergdec7}).

We now check that the maps $ED$ and $\Mes$ are both Borel measurable and are inverses of each other.
It is clear by definition that the map $ED$ is Borel measurable and that $ED\circ \Mes=Id$.
We proceed to the proof of the remaining claims.

First we check that the map $\Mes$ is Borel measurable. Indeed, take $\alpha_1, \alpha_2\in {\mathbb R}$,
take  a set $A\in {\mathcal B}(X)$ and consider the set ${\tilde A}_{\alpha_1, \alpha_2}\subset {\mathfrak M}(\mergtrho)$ given by the formula
$$
{\tilde A}_{\alpha_1, \alpha_2}=\{{\overline \nu}\in {\mathfrak M}(\mergtrho): {\overline \nu}\left(\{\eta\in \mergtrho: \eta(A)>\alpha_1 \}
\right)>\alpha_2
\}.
$$
It is clear that
$$
\left(\Mes\right)^{-1}\left( {\tilde A}_{\alpha_1, \alpha_2}\right)=
\{{\nu}\in \mtrho: { \nu}\left(\{x\in X: {\mathcal A}_{\infty}^{\rho}\chi_A(x)>\alpha_1 \}
\right)>\alpha_2
\},
$$
and measurability of
the map $\Mes$ is proved.

It remains to show that for a given measure $\nu\in\mtrho$ there is only one measure ${\overline \nu}\in {\mathfrak M}(\mergtrho)$
such that $\nu=ED({\overline \nu})$ --- namely, ${\overline \nu}=\Mes(\nu)$.
To prove this invertibility of the map $ED$ it suffices to establish the following
\begin{proposition}
Let ${\overline \nu}_1, {\overline \nu}_2\in {\mathfrak M}(\mergtrho)$. If  ${\overline \nu}_1\perp {\overline \nu}_2$, then also
$ED({\overline \nu}_1)\perp ED({\overline \nu}_2)$.
\end{proposition}
\begin{proof}
Let $\nu_0= ED(({\overline \nu}_1+{\overline \nu}_2)/2)$, and let $A_1, A_2\subset {\mathfrak M}(\mergtrho)$ be disjoint sets satisfying $$
{\overline \nu}_1(A_1)={\overline \nu}_2(A_2)=1; \ {\overline \nu}_1(A_2)={\overline \nu}_2(A_1)=0.
$$
The sets $X_1=\left(\Mes_{\xi^{\nu_0}}\right)^{-1}(A_1)$,  $X_2=\left(\Mes_{\xi^{\nu_0}}\right)^{-1}(A_2)$ are then disjoint and $\nu_0$-measurable. Furthermore, by definition we have
$$
ED({\overline \nu}_1)(X_1)=ED({\overline \nu}_2)(X_2)=1; \ ED({\overline \nu}_1)(X_2)=ED({\overline \nu}_2)(X_1)=0,
$$
whereby the Proposition is proved and the uniqueness of the ergodic decomposition is fully established.
\end{proof}

\section{Proof of Theorem \ref{ergdecinfpcl}}
In the proof of Corollary \ref{infdec} we have constructed an ergodic decomposition

\begin{equation}
\label{nuf2}
 \nu=\displaystyle\int\limits_{{\mathfrak M}^{\infty}_{f,1,erg}} \eta\, d\tilde \nu (\eta),
\end{equation}
where the measure $\tilde \nu\in {\mathfrak M}({\mathfrak M}^{\infty}_{f,1,\erg})$ is automatically admissible.

Given any positive measurable function $\varphi: {\mathbb P}{\mathfrak M}^{\infty}\to {\mathbb R}_{>0}$, we can deform the decomposition
(\ref{nuf2}) by writing
\begin{equation}
\label{nuf3}
 \nu=\displaystyle\int\limits_{{\mathfrak M}^{\infty}_{f,1,erg}} \frac{\eta}{\varphi({\bf p}(\eta))}\, \varphi({\bf p}(\eta))d\tilde \nu (\eta).
\end{equation}

Conversely, for any $\sigma$-finite measure $\nu^{\prime}\in {\mathfrak M}^{\infty}({\mathbb P}{\mathfrak M}^{\infty})$
satisfying $[\nu^{\prime}]=[{\bf p}_*{\tilde \nu}]$, we can immediately give a measure
${\check \nu}\in {\mathfrak M}^{\infty}({\mathfrak M}^{\infty})$ such that ${\bf p}_*{\check \nu}=\tilde\nu$ and
$$
\nu=\displaystyle\int\limits_{{\mathfrak M}^{\infty}} \eta\, d\check \nu (\eta).
$$

Since $\tilde\nu$ is admissible, the measure $\check \nu$ with the desired properties is clearly unique.

To complete the proof, we must now show that the measure class $[{\bf p}_*\overline \nu]$
is the same for all admissible measures $\overline \nu$ occurring
in the ergodic decomposition of the given measure $\nu$.

Recall that the map
$
P_f: \, \mathfrak M^\infty_f \longrightarrow \mathfrak M
$ is defined
by the formula
$$
P_f (\nu) \; = \; \frac{f\nu}{\nu(f)}.
$$
For $\lambda \in \mathbb R_+$ we clearly have
\[
P_f (\lambda \nu) \; = P_f(\nu).
\]
The map $P_f$ therefore induces a map from $\mathbb P\mathfrak M^\infty_f$ to $\mathfrak M$, for which we keep the same symbol.

The map $P_f: \mathbb P\mathfrak M^\infty_f \rightarrow \mathfrak M$ is invertible: the inverse is the map that to a measure $\nu \in \mathfrak M$ assigns the projective equivalence class of the measure $\frac{\nu}{f}$.

By definition, given any ergodic decomposition
\[
\nu \;=\; \int\limits_{\mathfrak M^\infty} \eta \, d \tilde \nu(\eta)
\]
of a measure $\nu \in \mathfrak M_f^\infty$, for the measure $\tilde \nu \in \mathfrak M( \mathfrak M^\infty)$ we have
\[
\tilde \nu ( \mathfrak M_{f, \erg}^\infty) \;=\;1.
\]
Take therefore an ergodic decomposition
\begin{equation} \label{ergdecnu}
\nu \;=\; \int\limits_{\mathfrak M_{\erg, g}^\infty} \eta\, d \tilde \nu (\eta).
\end{equation}
Applying the map $P_f$, write
\begin{equation} \label{pferg}
P_f \nu \;=\; \int\limits_{\mathfrak M_{\erg, f}^\infty} P_f \,\eta \cdot \frac{\eta(f)}{\nu(f)}\, d\tilde \nu(\eta).
\end{equation}
The measure
\[
\frac{\eta(f)}{\nu(f)} \, d\tilde \nu (\eta)
\]
is a probability measure on $\mathfrak M_{\erg,f}^\infty$ since so is $P_f \eta$ for any $\eta \in \mathfrak M_f^\infty$.

Introduce a measure $\check \nu \,\in\, \mathfrak M(\mathfrak M^\infty_{\erg, f})$ by the formula
\[
d \check \nu (\eta) \;=\; \frac{\eta(f)}{\nu(f)}\, d\tilde \nu(\eta)
\]
and rewrite (\ref{pferg}) as follows:
\begin{equation} \label{pfstar}
P_f \nu \;=\; \int\limits_{\mathfrak M} \eta \, d \left( ( P_f)_* \,\check \nu\right).
\end{equation}
By definition, the formula (\ref{pfstar}) yields an ergodic decomposition of the measure $P_f\nu \in \mathfrak M(\mathfrak T, \rho_f)$, indeed, the measure $(P_f)_*\, \check \nu$ is by definition supported on $\mathfrak M_{\erg} (\mathfrak T, \rho_f)$. Since ergodic decomposition is unique in $\mathfrak M(\mathfrak T, \rho_f)$, we obtain that the measure $(P_f)_*\, \check \nu$ does not depend on a specific initial ergodic decomposition (\ref{ergdecnu}).

From the clear equality $[\check \nu] = [\tilde \nu]$ it follows that
\[
[(P_f)_* \, \check \nu\,] \;=\; [(P_f)_* \, \tilde \nu \,],
\]
and, consequently, the measure class $[(P_f)_* \, \tilde \nu \,]$ does not depend on the specific choice of an ergodic decomposition (\ref{ergdecnu}).

Now recall that the map $P_f$ induces a Borel isomorphism between Borel spaces $\mathbb P\mathfrak M_f^\infty$ and $\mathfrak M$. Since the measure class $[(P_f)_* \, \tilde \nu \,]$ does not depend on the specific choice of an ergodic decomposition, the same is also true for the measure class $[ {\bf p_* \,\tilde \nu} \,]$. 
The Proposition is proved completely.

\subsection{Finite and infinite ergodic components.}
 
Ergodic components of an infinite $G$-invariant measure can be both finite and infinite, 
and the preceding results immediately imply the following description of 
the sets on which finite and infinite ergodic componets of an inifnite invariant measure are supported.

\begin{corollary}

Let $\mathfrak T$ be a measurable action of an inductively compact group $G$ on a standard Borel space $(X, \mathcal B)$, 
and let $\nu$ be a $\sigma$-finite $\mathfrak T$-invariant Borel measure on $X$ such that the space $L_1(X, \nu)$ contains a positive
Borel measurable fibrewise continuous function.
There exist two disjoint Borel $G$-invariant subsets $X_1$, $X_2$ of $X$ satisfying $X_1\cup X_2=X$ and such that 
the following holds.
\begin{enumerate}
\item There exists a family $Y_n$ of  Borel $G$-invariant subsets satisfying $\nu(Y_n)<+\infty$ and such that 
$$
X_1=\bigcup\limits_n Y_n.
$$
If $Y$ is a  Borel $G$-invariant subset  satisfying $\nu(Y)<+\infty$, then $\nu(X_1\setminus Y )=0$.
With respect to any ergodic decomposition, almost all ergodic components of the measure $\nu|_{X_1}$ are finite.

\item If $\varphi$ is a bounded measurable function, supported on $X_2$ and 
square-integrable with respect to $\nu$, then for the corresponding sequence of averages we have 
${\mathcal A}_n\varphi\to 0$ in $L_2(X, \nu)$. 
With respect to any ergodic decomposition, almost all ergodic components of the measure $\nu|_{X_2}$ are infinite.
\end{enumerate}
\end{corollary}
By definition, the sets  $X_1$, $X_2$ are unique up to subsets of measure zero.

In the case of continuous actions, a following description can also be given.
Let $X$ be a complete separable metric space, and let $\nu$ be a Borel measure that assigns finite weight to every ball.
Given a point $x\in x$, introduce the {\it orbital measures} $\eta^n_x$ by 
the formula
$$
\eta^n_x=\int\limits_{K(n)}\delta_{T_{k}x}d\mu_{K(n)}(k).
$$
Equivalently, for any bounded continuous function $f$ on $X$, we have 
$$
\int\limits_X fd\eta^n_x=\int\limits_{K(n)}f({T_{k}x})d\mu_{K(n)}(k).
$$

In this case the sets $X_1$, $X_2$ admit the following characterization:
the set $X_1$ is the set of all $x$ for which the sequence $\eta^n_x$ weakly converges 
to a probability measure as $n\to\infty$, while the set $X_2$ is the set of all $x$   
such that for any  bounded continuous function $f$ on $X$ whose support is a bounded set, we have 
$$
\lim\limits_{n\to\infty}\int\limits_X fd\eta^n_x=0.
$$

\section{Kolmogorov's example and proof of Proposition \ref{twoindec}.}

\subsection{Kolmogorov's Example.} For completeness of the exposition we briefly recall Kolmogorov's example \cite{Fomin} showing that, for actions of large groups, ergodic invariant probability measures may fail to be indecomposable.

Let $G$ be the group of all bijections of $\mathbb Z$, and let $\Omega_2$ be the space of bi-infinite binary sequences. The group $G$ acts on $\Omega_2$ and preserves any Bernoulli measure on $\Omega_2$.

Let $G_0 \subset G$ be the subgroup of {\it finite} permutations, that is, permutations that only move a finite subset of symbols. The group $G_0$ is inductively compact. De Finetti's Theorem states that $G_0$-invariant indecomposable (or, equivalently, ergodic) probability measures on $\Omega_2$ are precisely the Bernoulli measures.

It follows that $G$-invariant indecomposable probability measures are precisely Bernoulli measures as well. Nonetheless, if $\nu_1$ and $\nu_2$ are two distinct non-atomic Bernoulli measures on $\Omega_2$, then the measure $\frac{\nu_1 + \nu_2}{2}$ is ergodic! Indeed, the group $G$ has only countably many orbits on $\Omega_2$ and it is easily verified that any $G$-invariant set must have either full or zero measure with respect to $\frac{\nu_1+\nu_2}{2}$.

\subsection{Proof of Proposition \ref{twoindec}.}

As before, let $(X, \mathcal B)$ be a standard Borel space. Let $G$ be an arbitrary group, and let $\mathfrak T$ be an action of $G$ on $X$. The action $\mathfrak T$ will be called {\it weakly measurable} if for any $g \in G$ the transformation $T_g$ is Borel measurable.
Similarly, a positive multiplicative cocycle
\[
\rho: \; G\times X \longrightarrow \mathbb R_{>0}
\]
will be called {\it weakly measurable} if for any $g \in G$ the function $\rho( g,x)$ is Borel measurable in $x$. For a weakly measurable cocycle $\rho$ the space $\mathfrak M(\mathfrak T, \rho)$ is defined in the same way and is again a convex cone. A measure $\nu \in \mathfrak M(\mathfrak T, \rho)$ will be called { \it strongly indecomposable } if a representation
\[
\nu = \alpha \nu_1 + (1-\alpha) \nu_2
\]
with $\nu_1, \nu_2 \in \mathfrak M( \mathfrak T, \rho)$, $\alpha \in (0,1)$ is only possible when $\nu = \nu_1 = \nu_2$. A measure $\nu$ will be called {\it weakly indecomposable } if for any Borel measurable set $A$ satisfying, for every $g \in G$, the condition $\nu(A \triangle T_g A)=0$, we must have $\nu(A) = 0$ or $\nu(A) = 1$.

\vspace{5mm}
\begin{proposition}
A measure $\nu \in \mathfrak M ( \mathfrak T, \rho)$ is weakly indecomposable if and only if it is strongly indecomposable.
\end{proposition}
It is more convenient to prove the following equivalent reformulation.
\begin{proposition}
Let $\rho$ be a positive multiplicative weakly measurable cocycle over a weakly measurable action of a group $G$ on a standard Borel space $(X, \mathcal B)$. Let $\nu_1, \nu_2 \in \mathfrak M( \mathfrak T, \rho)$ be weakly indecomposable. Then either $\nu_1 = \nu_2$ or $\nu_1 \perp \nu_2$.
\end{proposition}
\begin{proof}
Indeed, let $\nu_1, \nu_2 \in \mathfrak M( \mathfrak T, \rho)$ be weakly indecomposable. Consider the Jordan decomposition of $\nu_1$ with respect to $\nu_2$ and write
\[
\nu_1 \; =\; \widetilde \nu_2 + \nu_3, \quad \widetilde \nu_2 \ll \nu_2, \quad \nu_3 \perp \nu_2.
\]

Since $\nu_2 \circ T_g \ll \nu_2$, we also have $\nu_2 ( \setminus T_gA)=0$ for each $g \in G$. It follows that for each $g \in G$ we have $\nu_1 (A \triangle T_g A) = 0$, whence either $\nu_1 (A) =0$ or $\nu_1 (A) = 1$. If $\nu_1(A)=0$, then $\nu_1 \perp \nu_2$, and we are done. If $\nu_1(A) = 1$, then $\nu_3=0$, and we have $\nu_1 \ll \nu_2$. Set
\[
\varphi \;=\; \frac{d\nu_1}{d\nu_2}.
\]
Since $\nu_1, \nu_2 \in \mathfrak M( \mathfrak T, \rho)$ and $\nu_1 \ll \nu_2$, for each $g \in G$ the function $\varphi$ satisfies, $\nu_2$-almost surely, the equality
\[
\varphi(T_g x) \;=\; \varphi(x)
\]
But then, the weak indecomposability of $\nu_2$ implies that $\varphi=1$ almost surely with respect to $\nu_2$, and, therefore, $\nu_1 = \nu_2$. The Proposition is proved completely.
\end{proof}

\end{document}